\documentclass[12pt,a4j]{amsart}
\usepackage{amsmath,amsthm,amssymb}
\usepackage{graphicx}

\usepackage{psfrag}\usepackage{varioref}%

\newtheorem*{theorema}{Main Theorem}

\newtheorem{prop}[subsection]{Proposition}
\newtheorem{lemma}[subsection]{Lemma}
\newtheorem{sublemma}[subsection]{Sublemma}

\newtheorem{cor}[subsection]{Corollary}

\begin{document}
\title{Asymptotic likelihood of chaos for
smooth families of circle maps}
\author{Hiroki Takahasi}
\address{Department of Mathematics, Kyoto University, Kyoto 606-8502
Japan}
 \email{takahasi@math.kyoto-u.ac.jp}
\begin{abstract}
We consider a smooth two-parameter family
$f_{a,L}\colon\theta\mapsto
\theta+a+L\Phi(\theta)$ of circle maps 
with a finite number of critical points. For sufficiently large
$L$ we construct a set $A_L^{(\infty)}$ of $a$-values of positive
Lebesgue measure for which the corresponding $f_{a,L}$ exhibits an
exponential growth of derivatives along the orbits of the critical
points. Our construction considerably improves the previous one of
Wang and Young for the same class of families, in that the
following asymptotic estimate holds: the Lebesgue measure of
$A_L^{(\infty)}$ tends to full measure in $a$-space as $L$ tends
to infinity.
\end{abstract} \maketitle

\tableofcontents

\section{Introduction}
In the study of one-dimensional dynamical systems, one important
question is: {\it how often are dynamics chaotic?} Here, "often"
should be understood in the sense of Lebesgue measure in parameter
space and "chaotic dynamics" corresponds to maps which have
absolutely continuous invariant probability measures (acip for
short).

In this direction, there is a huge gap between a general belief
and the existing theory. For the quadratic family $x\mapsto
1-ax^2$, for instance, it is believed, and also suggested by
rigorous computations \cite{Tu} that the set of parameters
corresponding to acips should have large Lebesgue measure.
Meanwhile, what is presently known at best is that this set has
positive, yet very small Lebesgue measure \cite{Jak81}.
The aim of this paper is to narrow this gap, for certain smooth
families of maps on the circle.

Let $\Phi\colon S^1=\mathbb R/\mathbb Z\to\mathbb R$ be a Morse
function.
 We consider a two-parameter family of circle maps of the form
$$f_{a,L}\colon\theta\mapsto \theta+a+L\Phi(\theta)\quad a\in[0,1), L>0.$$
The family with $\Phi(\theta)=\sin(2\pi\theta)$ was introduced by
Arnol'd \cite{A65} and played an important role in the creation of
KAM theory. For small $L$, the map is a diffeomorphism, and this
case was intensively studied for its connection with
quasi-periodic motions on invariant tori in conservative systems.
We explore dynamics at the other end of the spectrum, namely, the
case with sufficiently large $L$. Then the map has a finite number
of critical points. Its graph has large slopes outside of a small
neighborhood of the critical points.

The family of circle maps with large $L$ becomes important in the
theory of {\it rank-one strange attractors}, developed by Wang and
Young \cite{WY0} \cite{WY1}, based on the fundamental works of
Jakobson \cite{Jak81}, Benedicks and Carleson \cite{BC91}, Mora
and Viana \cite{MV93}, and others. In brief terms, the theory
indicates that dynamics of strange attractors in certain
physically relevant multi-dimensional systems may be partially
understood by analyzing the above circle family with large $L$.
Indeed, the existence of strange attractors in certain
periodically forced ODEs with fully stochastic behaviors was
rigorously proved along this line \cite{WY1}.


For large $L$, a positive measure set of values of $a$ was
constructed in \cite{WY1} corresponding to maps with a unique
acip. However, their construction seems far from optimal in that
$\liminf_{L\to\infty}{\rm Leb }\left(\{a\in[0,1)\colon
f_{a,L}\text{ has an acip }\}\right)>0$ does not follow.
Meanwhile, it is physically relevant to consider what happens in
the asymptotic case $L\to\infty$. An intuition is that parameters
with acip are in abundance.

The reason for this deficiency is that their construction has to
start with very small parameter intervals containing "good
parameters", and the dependence of the sizes of the intervals on
$L$ is unclear. In this paper we develop another argument and show
that $\lim_{L\to\infty}{\rm Leb}\left(\{a\in[0,1)\colon
f_{a,L}\text{ has a unique acip }\}\right)=1.$ A key ingredient is
to notice that for sufficiently large $L$ it is possible to carry
out an inductive construction taking the whole parameter space
$[0,1)$ as a start-up interval.



\subsection{Statement of the result}
Let $C$ denote the set of critical points of $f_{a,L}$, which does
not depend on $a$. Since all the critical points of $\Phi$ are
non-degenerate, the cardinality of $C$ is constant for all large
$L$. For each $c\in C$, all $a\in[0,1)$ and $i\geq0$, write
$c_i(a)$ for $f_a^{i+1}c$. Let $|\cdot|$ denote the
one-dimensional Lebesgue measure.

Our main theorem states the abundance of parameters for which the
derivatives along orbits of all the critical points grow
exponentially fast under iteration. It is well-known \cite{CE}
that this growth condition implies the existence of acips.
\begin{theorema}
For the above $(f_{a,L})$ there exists $\lambda>0$ such that for
all sufficiently large $L$ there exists a set $A_L^{(\infty)}$ in
$[0,1)$ with positive Lebesgue measure such that for all $a\in
A_L^{(\infty)}$ and each $c\in C$, $|(f_{a,L}^n)'c_0|\geq
L^{\lambda n}$ holds for every $n\geq0$. Moreover
$\lim_{L\to\infty}|A_L^{(\infty)}|=1$ holds.
\end{theorema}

For parameters in $A_L^{(\infty)}$, following \cite{WY0}
\cite{WY1} it is possible to construct a {\it unique} acip $\mu$
for which Lebesgue almost every $\theta\in S^1$ is generic, that
is,
$$\lim_{n\to\infty}\frac{1}{n}\sum_{i=0}^{n-1}\varphi(f_{a,L}^i\theta)
=\int \varphi d\mu\quad\text{for all continuous }\varphi\colon
S^1\to\mathbb R.$$ In particular, parameters corresponding to maps
with periodic attractors are contained in the complement of
$A_L^{(\infty)}$. The same statement also follows from directly
showing that all periodic points are hyperbolic repelling for
parameters in the theorem.

Our construction yields an explicit measure estimate in terms of
$L$ (Proposition \ref{measure}), and as a by-product gives a bound
on the speed of the convergence $|A_L^{(\infty)}|\to1$ as follows:
the measure of the complement of $A_L^{(\infty)}$ decreases to
zero as $L\to\infty$, faster than any power of $L^{-1}$.

\subsection{Outline of a proof}
In the context of one-dimensional maps with critical points, the
existence of acips for a positive measure set of parameters was
first proved by Jakobson \cite{Jak81}. See \cite{BC85} \cite{G87}
\cite{Ry} \cite{BC91} \cite{T93a} \cite{T93b} \cite{MS} \cite{TTY}
\cite{Y99} \cite{S03} for alternative arguments and
generalizations.

An outline of the proof of the theorem is similar in spirit to
that of \cite{Jak81}, and (hence) to those of all the subsequent
papers. The positive measure set $A_L^{(\infty)}$ is constructed
by induction: at each step we get rid of undesirable parameters
for which the corresponding maps may not have acips. In doing this
we bring together ideas from Benedicks and Carleson \cite{BC85}
\cite{BC91}, Tsujii \cite{T93a} \cite{T93b}, and develop them
further.


Key constants are $\sigma,\lambda,\alpha,N,L$, chosen in this
order. we have $\sigma,\alpha\ll1$ and $N,L\gg1$. The choice of
them are made explicit afterwards.

Our induction scheme is divided into two parts. The first part
consists of finite steps $0,1,2,\cdots,N$. The second part
consists of the remaining steps $N+1,N+2,\cdots$. At each step $n$
we construct a set $A^{(n)}$. The parameter set in the main
theorem is given by $A^{(\infty)}=\bigcap_{n\geq0}A^{(n)}$.

For $0\leq n\leq N$, let $A^{(n)}$ denote the set of all
$a\in[0,1)$ such that:
\begin{equation}\label{mis}
d(c_i(a),C)\geq\sigma \text{ for every }i\in[0,n] \text{ and }c\in
C.\end{equation} 
It turns out that parameters in $A^{(N)}$ enjoy a uniformly
expanding property outside of a small neighborhood of the critical
points of fixed size (Corollary \ref{outside}). Since our primary
interest is an exponential growth of derivatives, this property
permits us to concentrate on returns of critical orbits to the
inside of this small neighborhood.

Condition (\ref{mis}) for every $n$ is satisfied only for
parameters in a set with zero Lebesgue measure. To get a set of
positive measure we need to relax this condition. For each $n\geq
N$
and $c\in C$ 
we introduce two conditions:
\begin{itemize}
\renewcommand{\labelitemi}{$(X)_{c,n}$}
\item
$|(f_a^{j-i})'c_i|\geq L
\cdot\min\{\sigma, L^{-\alpha i}\}\text{ for every }0\leq i<j\leq
n;$ 
\renewcommand{\labelitemi}{$(Y)_{c,n}$}
\item
$|(f_a^i)'c_0|\geq L^{\lambda i}\text{ for every }0\leq i\leq n.$
\end{itemize}
We say $f_a$ satisfies $(X)_n$ if $(X)_{n,c}$ holds for each $c\in
C$. The meaning of $(Y)_n$ is analogous. There two conditions are
taken as assumptions of induction at step $n$.

To recover the assumptions of induction at the next step $n+1$, we
exclude from further consideration all parameters in $A^{(n)}$ for
which some analytic condition leading to $(X)_{n+1}$ $(Y)_{n+1}$
fails. This condition is introduced in section 5. The remaining
parameters constitute $A^{(n+1)}$.


This paper is organized as follows. In section 2 we prove three
lemmas which will be frequently used later. In section 3 we
perform the first part of the inductive steps $0,1,\cdots,N$ and
estimate the measure of $A^{(0)}$, $A^{(1)},\cdots,A^{(N)}$. In
section 4 we establish a common technique, a recovery of
expansion, which will be used to estimate the measure of $A^{(n)}$
for $n>N$. In section 5 we introduce condition $W_n$ which defines
the set $A^{(n)}$ for $n>N$. In section 6 we estimate the measure
of $A^{(\infty)}.$

Unless otherwise stated, we always assume that $L$ is sufficiently
large.


\section{Fundamental lemmas}
We prove three lemmas which will be frequently used later. Lemma
\ref{dist} gives distortion bounds for iterations of one fixed
map. Lemma \ref{samp} gives distortion bounds for critical values
for different parameters. Proofs are similar, by virtue of Lemma
\ref{trans} which asserts a similarity between space and parameter
derivatives, allowing us to transfer estimates in phase space to
parameter space.

Write $f$ for $f_{a,L}$.
 Let $C_{\varepsilon}$ denote the $\varepsilon$-neighborhood of
$C$. There exists $K_0\geq1$ depending only on $\Phi$ and small
$\varepsilon>0$ such that for all large $L$ we have:
\begin{align*}
&K_0^{-1}L|c-\theta|^2\leq|f(c)-f(\theta)|\leq K_0L |c-\theta|^2
\quad\text{ for }c\in C\text{
and }\theta\in C_\varepsilon;\\
&K_0^{-1}L|c-\theta|\leq|f'\theta|\leq K_0L
|c-\theta|\quad\text{ for }c\in C\text{ and }\theta\in C_\varepsilon;\\
&|f'|,|f''|\leq K_0L.
\end{align*}

\subsection{Distortion in phase space}
For $\theta\in S^1$, $n\geq1$ and $i\in[0,n-1]$, define
\begin{equation}\label{Theta}
d_i(a,\theta)= |(f_a^i)'\theta|^{-1}\cdot|f_a'(f_a^i\theta)|,
\end{equation}
when it makes sense. Fix $\beta\in\left(\frac{3}{2},2\right)$ and
define
$$D_n(a,\theta)=L^{-\beta}\cdot\left[\sum_{0\leq i\leq n-1}
d_i^{-1}(a,\theta)\right]^{-1}.$$ Put
$$K=\exp\left(2K_0L^{1-\beta}\right).$$
Note that $K\to1$ as $L\to\infty$.

\begin{lemma}\label{dist}
If $\theta\in S^1$, $n\geq1$ and $f_a^i\theta\notin C$ holds for
every $0\leq i\leq n-1$, then
$$\frac{|(f_a^n)'\varphi|}{|(f_a^n)'\psi|}\leq K\quad\text{ for all }
\varphi,\psi\in[\theta-D_n(a,\theta),\theta+D_n(a,\theta)].$$
\end{lemma}

\begin{proof}
Write $f$, $d_i$, $D_n$ for $f_a$, $d_i(a,\theta)$,
$D_n(a,\theta)$ correspondingly. Let $I=[\theta-D_n,\theta+D_n].$
It suffices to prove the following for $j=0,\cdots,n-1:$
\begin{equation}\label{disteq1}
|f^jI|\sup_{f^jI}\frac{|f''|}{|f'|}\leq \log K\cdot
d_j^{-1}\left[\sum_{0\leq i\leq
n-1}d_i^{-1}\right]^{-1}.\end{equation} Indeed, summing this over
all $j=0,1,\cdots,n-1$ gives
\begin{align*}
\sup_{\theta,\psi\in I}\log
\frac{|(f^n)'\theta|}{|(f^n)'\psi|}&\leq\sum_{0\leq j\leq
n-1}\sup_{\theta,\psi\in I}\log \frac{|f'(f^{j}\theta)|}
{|f'(f^{j}\psi)|}\\
&\leq\sum_{0\leq j\leq n-1}|f^jI|\sup_{
f^jI}\frac{|f''|}{|f'|}\\
&\leq \log K.\end{align*}

We prove (\ref{disteq1}) by induction on $j$. We only give a proof
for the general step of the induction. A proof for the initial
step (the case $j=0$) is completely analogous.

Summing (\ref{disteq1}) over all $j=0,1,\cdots, k-1$ implies
$\frac{|(f^k)'\varphi|}{|(f^k)'\psi|}\leq K$ for all
$\varphi,\psi\in I.$ Hence
\begin{align*}|f^kI|&\leq K|(f^k)'\theta||I|\\
&= KD_nd_k^{-1}|f'(f^k\theta)|\quad\text{by the definition of
$d_k$ }\\
&\leq KL^{-\beta}|f'(f^k\theta)|d_k^{-1} \left(\sum
d_i^{-1}\right)^{-1}\quad\text{by the definition of $D_n$
}\\
&\leq KL^{-\beta}|f'(f^k\theta)|.\end{align*} Hence, for all
$\psi\in I$ we have
\begin{align*}|f'(f^k\theta)-f'(f^k\psi)|&\leq
K_0L|f^kI|\leq KK_0L^{1-\beta}|f'(f^k\theta)|.
\end{align*}
This yields $|f'|\geq (1-KK_0L^{1-\beta})|f'(f^k\theta)|$ on
$f^kI$, and therefore
\begin{align*}
|f^kI|\sup_{f^kI}\frac{|f''|}{|f'|} &\leq
\frac{KK_0L^{1-\beta}}{1-KK_0L^{1-\beta}}\cdot
d_{k}^{-1}\left(\sum d_i^{-1}\right)^{-1}\\
&\leq 2K_0L^{1-\beta}\cdot d_{k}^{-1}\left(\sum
d_i^{-1}\right)^{-1},
\end{align*}
where the last inequality is because of $K\to1$ as $L\to\infty$.
\end{proof}

\subsection{Transversality}
For $c\in C$ and $i\geq0$, write $c_i'(a)= \frac{dc_i}{da}(a)$.
\begin{lemma}\label{trans}
Let $f_a$ satisfy ${(Y)}_{n,c}$. Then we have
$$1-L^{-\lambda/2}\leq\frac{\left|c_{n}'(a)\right|}
{\left|(f_a^n)'(c_0(a))\right|}\leq 1+L^{-\lambda/2}.$$
\end{lemma}
\begin{proof}
Since $c'_{n}(a)=1+f_a'(c_{n-1}(a))c_{n-1}'(a),$ we have
$$c_{n}'(a)=1+f_a'(c_{n-1})+f_a'(c_{n-1})
f_a'(c_{n-2})+\cdots+f_a'(c_{n-1})f_a'(c_{n-2}) \cdots
f_a'(c_{1})f_a'(c_{0}).$$ Dividing both sides by
$(f_a^n)'(c_0)=f_a'(c_{n-1})f_a'(c_{n-2})\cdots
f_a'(c_{1})f_a'(c_{0})$ gives
$$\frac{c_n'(a)}{(f_a^n)'c_0}=1+\sum_{i=1}^{n}\frac{1}
{(f^i)'(c_0)}.$$ Hence
$$1-\sum_{i=1}^{n}
\frac{1}{|(f^i)'c_0|}\leq
\left|\frac{c_n'(a)}{(f_a^n)'c_0}\right|\leq1+
\sum_{i=1}^{n}\frac{1}{|(f^i)'c_0|}.$$ ${(Y)}_{n,c}$ yields the
desired inequality.
\end{proof}

\subsection{Distortion in parameter space}
We transfer the distortion estimate in Lemma \ref{dist} to
parameter space.
 For $c\in C$ and $n\geq1$, Define
$$\hat\Delta_n(a_*,c)=[a_*-D_n({a_*},c_0(a_*)),
a+D_n({a_*},c_0(a_*))].$$ Let
$$K'=\exp\left(L^{-\frac{1}{4}}+3\right).$$

\begin{lemma}\label{samp}
Let $f_{a_*}$ satisfy ${(X)}_{n,c}$ and ${(Y)}_{n,c}$. Then
$$\frac{|c_n'(a)|}{|c_n'(b)|}\leq K'\quad\text{for all }
a,b\in\hat\Delta_n(a_*,c).$$
\end{lemma}
\begin{proof}


Write $d_i$, $D_n$, $\hat\Delta_n$ for $d_i(c_0(a_*))$,
$D_n(c_0(a_*))$, $\hat\Delta_n(a_*,c)$ correspondingly. We argue
by induction on $k\in[0,n-1]$, with the assumption that
$\frac{|c_j'(a)|}{|c_j'(b)|}\leq K'$ holds for all $a,b\in\Delta$
and $j=0,1,\cdots,k$.  Note that this assumption for $k=0$ is
trivially satisfied by $c_0'=1$.
\begin{sublemma}\label{sub}
For all $j=0,1,\cdots,k$ we have
\begin{equation*}\log\frac{|(f_{a}')c_j(a)|} {|(f_{b}')c_j(b)|}
\leq \frac{2K_0K'L^{1-\beta}}{1-2K_0K'L^{1-\beta}}\cdot
d_j^{-1}\cdot d_j^{-1}\left[\sum_{0\leq i\leq
n-1}d_i^{-1}\right]^{-1}\quad\text{for all
}a,b\in\hat\Delta_n(a_*,c).\end{equation*}
\end{sublemma}
\begin{proof}
We have
\begin{align*}\left|f_{a}'c_j(a)-f_{b}'c_j(b)\right|
&\leq\left|f_{a}'c_j(a)-f_{b}'c_j(a)\right|+
\left|f_{b}'c_j(a)-f_{b}'c_j(b)\right|\\
&=\left|f_{b}'c_j(a)-f_{b}'c_j(b)\right|\quad\text{since
$f_a'\theta-f_b'\theta=0$}\\
&\leq K_0L|c_j(\Delta)|\quad\text{since $|f''|\leq
K_0L$}.\end{align*} Using the assumption of induction and Lemma
\ref{trans},
\begin{align*}
|c_j(\Delta)| &\leq
2K'|(f_{a_*}^j)'c_0(a_*)|d_j d_j^{-1}D_n\\
&= 2K'L^{-\beta}|(f_{a_*}')c_j(a_*)|d_j^{-1}\left(\sum d_i^{-1}\right)^{-1}\\
&\leq 2K'L^{-\beta}|(f_{a_*}')c_j(a_*)|.
\end{align*}
These two inequalities imply the desired inequality.
\end{proof}

If $j\geq1$, Lemma \ref{trans} and $(Y)_{n,c}$ for $f_{a_*}$ give
$$|c_j'(a)|\geq K'^{-1}|c_j'(a_*)|\geq \frac{K'^{-1}}{2}
|(f_{a_*}^j)'c_0(a_*)|\geq L^{\frac{\lambda j}{2}}.$$ Since
$|c_0'|=1$, the same inequality remains valid if $j=0$. We have
\begin{equation*}
\left|\frac{c_{j+1}'(a)}{c_{j}'(a)}- (f_a')c_{j}(a)\right|
=\left|\frac{1}{c_{j}'(a)}\right|\leq L^{-\frac{\lambda j}{2}}.
\end{equation*}
The first factor in the right hand side of Sublemma \ref{sub} goes
to $0$ as $L\to0$. Using this and $(X)_{n,c}$ give
$$|(f_{a})'c_j(a)| \geq\frac{1}{2}|(f_{a_*})'c_j(a_*)| \geq \frac{
L}{2}\cdot\min\{\sigma,L^{-\alpha j} \}.$$
Hence we have
\begin{equation*}\label{sad}
 \frac{|c_{j+1}'(a)|}
{|c_j'(a)|}\geq  \frac{L}{2}\cdot\min\{\sigma,L^{-\alpha j}\}-
L^{-\frac{\lambda j}{2}}\geq \frac{1}{3}L^{-\alpha
j}.\end{equation*} These three inequalities imply
\begin{equation*}
\left|\log\frac{|c_{j+1}'(a)|}{|c_j'(a)|}-\log|(f_{a})'c_j(a)|\right|\leq
L^{-\frac{\lambda j}{3}}.\end{equation*} By Sublemma \ref{sub},
\begin{equation*}\left|\log\frac{|c_{j+1}'(a)|}
{|c_j'(a)|}- \log|(f_{a_*})'c_j(a_*)|\right| \leq
L^{-\frac{\lambda j}{3}}+\frac{3}{2}\cdot d_j^{-1}\left(\sum
d_i^{-1}\right)^{-1}.
\end{equation*}
Thus, for all $a$, $b\in \hat\Delta_n$,
\begin{equation*}
\left|\log\frac{|c_{j+1}'(a)|}{|c_{j+1}'(b)|}-
\log\frac{|c_{j}'(a)|}{|c_{j}'(b)|}\right| \leq 2L^{-\frac{\lambda
j}{3}}+ 3d_j^{-1}\left(\sum d_i^{-1}\right)^{-1}.
\end{equation*}
Summing this over all $j=1,\cdots,k$ implies
$\frac{|c_{k+1}'(a)|}{|c_{k+1}'(b)|}\leq K',$ which restores the
assumption of the induction.
\end{proof}

\section{Parameter exclusion: special steps}
In this section we perform the fist part of the induction from
step $0$ to $N$. Recall that the parameter sets $A^{(n)}$ for
$0\leq n\leq N$ are given by (\ref{mis}), where we set
\begin{equation}\label{sigma}
\sigma=K_0L^{-1+\frac{\beta}{2}}.
\end{equation}
We estimate the measure of $A^{(n)}$ for $n=0,1,\cdots, N$.

\subsection{Expansion} We consider orbits of critical points which stay
outside of $C_\sigma$. Lemma \ref{trans} transmits the expansion
along these orbits to parameter space. The next lemma asserts that
this expansion is large enough for critical values to completely
wrap $S^1$.
\begin{lemma}\label{wrap0}
Let $n\leq N$, $a\in[0,1)$ and $c\in C$. If $c_i(a)\notin
C_\sigma$ holds for $i=0,1,\cdots,n-1$, then
$c_{n}(\hat\Delta_n(a,c))=[0,1)$.
\end{lemma}

\begin{proof}
We have
\begin{align*}
|c_{n}(\hat\Delta_n(a,c))|&\geq {K'}^{-1}\left| c_{n}'(a)\right|
|\hat\Delta_n(a,c)|
\quad\text{by Lemma \ref{samp}}\\
&\geq\frac{{K'}^{-1}}{1+L} |(f_a^{n})'c_0|
|\hat\Delta_n(a,c)|\quad\text{by Lemma \ref{trans}}.\end{align*}
The assumption gives $ |(f_a^{n})'c_0|\geq
(L\sigma)^{n-i}|(f_a^i)'c_0|$, which yields $|(f^{n})'c_0|d_i\geq
(L\sigma)^{n-i}L\sigma$. Therefore
\begin{align*}|(f^{n})'c_0|^{-1}
|\hat\Delta_n(a,c)|^{-1}&=L^{\beta}\sum_{0\leq i\leq
n-1}|(f^{n})'c_0|^{-1}d_i^{-1}\\&\leq
\sigma^{-1}L^{-1+\beta}\sum_{i\geq1
}(L\sigma)^{-i}\\
&\leq\frac{\sigma^{-1}L^{-1+\beta}}{L\sigma-1}.
\end{align*}
Substituting this into the above inequalities we obtain the
desired inequality.
\end{proof}

\subsection{Amending the definition of parameter intervals}
We now introduce parameter intervals which are central to our
scheme. Let $c\in C$, $n\geq1$, and let $f_a$ satisfy $(X)_{n,c}$,
$(Y)_{n,c}$. As a result of Lemma \ref{trans}, the map $a\in
\hat\Delta_n(a,c)\mapsto c_n(\hat\Delta_n(a,c))$ may not be
injective (we have just shown that this is indeed the case if
$n\leq N$). This causes silly combinatorial problems. To deal with
this we shorten the interval as follows. For a compact interval
$I$ centered at $a$ and $r\in(0,1)$, let $r\cdot I$ denote the
interval centered at $a$ with length $r\cdot|I|$. Define
$$\Delta_n(a,c)=\begin{cases}
&\hat\Delta_n(a,c)\quad\text{if }
|c_n(\hat\Delta_n(a,c))|\leq\frac{1}{3};\\
&\frac{1}{9|c_n(\hat\Delta_n(a,c))|}\cdot\hat\Delta_n(a,c)\quad
\text{otherwise}\end{cases}.$$ By definition, $c_n(\Delta_n(a,c))$
is strictly contained in the half circle centered at $c_n(a)$.

\begin{prop}\label{initial}
For any sufficiently large $N$ there exists $L_0$ such that if
$L\geq L_0$, then for $n=0,1,\cdots,N$ we have
$|A^{(n)}|\geq(1-{^3\sqrt{\sigma}})^{n+1}$.
\end{prop}

\begin{proof}
The inequality for $n=0$ follows from the identity
$f_a\theta-f_b\theta=a-b$. We argue by induction on $n$. For
$(c,\tilde c)\in C\times C$, let
$$B_n(c,\tilde c)=\left\{a\in A^{(n-1)}\setminus A^{(n)}\colon
d(c_n(a),\tilde c)\leq\sigma\right\}.$$ We have
$A^{(n-1)}\setminus A^{(n)}=\bigcup B_n(c,\tilde c)$, where the
union runs over all $(c,\tilde c)\in C\times C$.

\begin{lemma}\label{intersect2}
Let $a$, $b\in B_n(c,\tilde c)$ and assume that (i)
$\Delta_n(a,c)\cap \Delta_n(b,c)\neq\emptyset$; (ii) $b\notin
\Delta_n(a,c)$. Then we have $\Delta_n(a,c)\subset\Delta_n(b,c)$.
\end{lemma}

\begin{proof}
Since $a\in B_n(c,\tilde c)$ we have $|c_n(a)-\tilde c|
\leq\sigma$. This and Lemma \ref{wrap0} together imply the
existence of a unique parameter $\bar a\in
\sqrt{\sigma}\cdot\Delta_n(a,c)$ such that $c_n(\bar a )=\tilde
c$. For the same reason, there exists a unique $\bar b\in
\sqrt{\sigma}\cdot\Delta_n(b,c)$ such that $c_n(\bar b)=\tilde c$.
Since the map $\hat a\in \Delta_n(a,c)\cup\Delta_n(b,c)\to
c_n(\hat a)$ is injective (by the above amendment), we have $\bar
a=\bar b$. This and the assumption (ii) implies that one of the
connected component of
$\Delta_n(a,c)-\sqrt{\sigma}\cdot\Delta_n(a,c)$ is contained in
$\sqrt{\sigma}\cdot\Delta_n(b,c)$, and thus the inclusion holds.
\end{proof}

Let $a_1\in B_n(c,\tilde c)$. We define a finite sequence
$a_1,a_2,\cdots$ in $B_n(c,\tilde c)$ inductively as follows.
Given $a_1,\cdots,a_k$, if $B_n(c,\tilde c
)\subset\bigcup_{j=1}^k\Delta_n(a_j,c)$, then we complete the
definition. Otherwise, choose $a_{k+1}\in B_n(c,\tilde c
)-\bigcup_{j=1}^{k}\Delta_n(a_j,c)$ arbitrarily. By (\ref{mis}),
the lengths of the intervals $\{\Delta_n(a_j,c)\}$ are uniformly
bounded from below, and thus the definition makes sense.

By Lemma \ref{intersect2}, any two of the intervals thus defined
are either disjoint or nested and altogether cover $B_n(c,\tilde c
)$. Moreover, by Lemma \ref{samp} and Lemma \ref{wrap0}, the set
$\Delta_n(a_j,c)-\sqrt{\sigma} \cdot\Delta_n(a_j,c)$ does not
intersect $B_n(c,\tilde c)$. Hence
$$|B_n(c,\tilde c)| \leq\sqrt{\sigma}\cdot\sum_{*}|\Delta_n(a_j,c)|,$$ where
$\sum_*$ denotes the summation over the maximal set of the
subscript $j$ for which the corresponding intervals are pairwise
disjoint. Since $|\Delta_n(a_j,c))|\leq L^{-\beta}$,
\begin{align*}
\sum_{*}|\Delta_n(a_j,c)|&\leq
\frac{|A^{(n-1)}|}{|A^{(n-1)}|}(1+2L^{-\beta})\\&\leq
(1+2L^{-\beta})\frac{|A^{(n-1)}|}{(1-{^3\sqrt\sigma})^{n}}\\
&\leq2|A^{(n-1)}|,
\end{align*}
where the last inequality holds for sufficiently large $L$
depending on $N$. This yields $|B_n(c,\tilde c)|\leq
2\sqrt\sigma|A^{(n-1)}|.$ Hence we have $$|A^{(n-1)}\setminus
A^{(n)}|\leq(\sharp C)^2\cdot 2\sqrt\sigma |A^{(n-1)}| \leq
{^3\sqrt\sigma}|A^{(n-1)}| ,$$ or
$|A^{(n)}|\geq(1-{^3\sqrt\sigma})|A^{(n-1)}|.$ This restores the
assumption of the induction.
\end{proof}

\section{Recovering expansion}
For large $L$, the dynamics is uniformly expanding in most of the
phase space, outside of a small neighborhood of the critical set
$C$. Meanwhile, returns of typical orbits to the inside of this
neighborhood are inevitable. In this section we deal with the loss
of derivatives associated with these returns to keep the further
evolution of derivatives in track.

\subsection{Constants}
We introduce several constants. Let
\begin{equation}\label{lambda0}
\lambda_0=\frac{1}{2}-\frac{\beta}{4}\quad\text{ and
}\quad\lambda=\frac{\lambda_0}{9}.\end{equation} Fix
$\alpha\ll\lambda$ and let
\begin{equation}\label{delta0}
\delta_0=L^{-1+\lambda_0}\quad\text{ and }\quad\delta=L^{-\alpha
N}.\end{equation}

We use three different sizes of neighborhoods of critical points,
given by $\sigma,\delta,\delta_0$. We have $\delta/\delta_0$,
$\delta_0/\sigma\to0$ as $L\to\infty$. The last convergence
follows from $\beta<2$.

\subsection{Bound periods}
Let $c\in C$, and let $f$ satisfy $(X)_{n,c}$, $(Y)_{n,c}$. For
$p\in[1,n-1]$, let
$$I_{p}(c)=\left(c+
\sqrt{L^{-1}D_{p+1}(c_0)},c+ \sqrt{L^{-1}D_{ p}(c_0)}\right].$$
 Let $I_{-p}(c)$
denote the interval which is the mirror image of $I_{p}(c)$ with
respect to $c$. We call $p$ a {\it bound period} for $\varphi\in
I_{p}(c)\cup I_{-p}(c)$. According to Lemma \ref{dist}, the orbit
of $f(\varphi)$ shadows the orbit of $c_0$ for $p$ iterates, with
bounded distortion.

We claim that for all $a\in A^{(N)}$, the intervals
$\{I_{p}(c)\colon 1\leq |p|\leq N,c\in C\}$ altogether cover
$C_{\delta_0}-C_\delta$. Indeed we have
\begin{align*}
L^{-1}D_{N}(c_0)&\leq |(f^{N-1})'c_0|^{-1}
|f'c_{N-1}|\quad\text{by the definition of $D_N(c_0$)}\\
&\leq (K_0^{-1}L\sigma)^{-N+1}K_0L\quad\text{by (\ref{mis})}\\
&<\delta^2\quad\text{by (\ref{sigma}) and (\ref{delta0}).}
\end{align*}
On the other hand,
\begin{align*}
\delta_0^2&= L^{-\beta+2\lambda_0}\sigma^2\quad\text{by (\ref{sigma})}\\
&\leq L^{-\beta}\sigma\quad\text{by (\ref{lambda0})} \\
&\leq L^{-1}D_1(c_0)\quad\text{by $\kappa>1$ and
$d(c_0,C)\geq\sigma$ in (\ref{mis}).}
\end{align*}
Hence $\sqrt{L^{-1}D_{N}(c_0)}<\delta< \delta_0 <
\sqrt{L^{-1}D_1(c_0)}$, which implies the claim.

\begin{lemma}\label{reclem1}
Let $f$ satisfy ${(X)}_{n,c}$, ${(Y)}_{n,c}$ for some $n\geq N$
and $c\in C$. Then for $ p\in[1,n-1]$ and $\varphi\in I_{p}(c)\cup
I_{- p}(c)$ we have:
\begin{itemize}
\renewcommand{\labelitemi}{$\text{(a)}$}
\item $\log|c-\varphi|^{-\frac{1}{\log L}}\leq p\leq\log |c-\varphi|^{-\frac{2}{\lambda\log L}}$;
\renewcommand{\labelitemi}{$\text{(b)}$}
\item
$|(f^{p+1})'\varphi|\geq\max\{|c-\varphi|^{-1+\frac{7\alpha}{\lambda}},
L^{\lambda(p+1)/3}\};$
\renewcommand{\labelitemi}{$\text{(c)}$}
\item if $p\leq N$, then
$|(f^{p+1})'\varphi|\geq L^{\lambda_0(p+1)/3}.$
\end{itemize}

\end{lemma}

\begin{proof}
We use the notation $\mathcal O(1)$ to denote all constants which
stay bounded and bounded away from zero as $L\to\infty$.

Since $|c_0-f\varphi|\leq K_0L |c-\varphi|^2$, we have
\begin{align*} |c_p-f^{p+1}\varphi|&\leq \mathcal O(1)L
|\theta-\varphi|^2|(f^p)'c_0|\quad\text{by Lemma \ref{dist}}\\
&\leq \mathcal O(1)D_p(c_0)|(f^p)'c_0|\quad \text{since
$\varphi\in I_{-p}(c)\cup I_{p}(c)$}.
\end{align*}
To estimate the right hand side we need
\begin{sublemma}\label{distrem1}
We have $|(f^n)'\theta|D_n(\theta)\leq K_0^2L^{2-\beta}.$
\end{sublemma}
\begin{proof}
We have
\begin{align*}|(f^n)'\theta|d_{n-1}d_{n-1}^{-1}D_n
&\leq L^{-\beta}|(f^n)'\theta|d_{n-1}\\
&= L^{-\beta}|f'(f^{n-1}\theta)|^2\\
&\leq K_0^2L^{2-\beta}.\end{align*}
\end{proof}
Hence we have $|c_p-f^{p+1}\varphi|\leq \mathcal O(1)L^{2-\beta}$.
On the other hand, since $|c_0-f\varphi|\geq K_0^{-1}L
|c-\varphi|^2$ we have
\begin{align*}
|c_p-f^{p+1}\varphi|&\geq
\mathcal O(1)L|c-\varphi|^2|(f^p)'c_0|\\
&\geq \mathcal O(1)L|c-\varphi|^2L^{\lambda p} \quad\text{by $
(Y)_{n,c}$.}
\end{align*}
Putting these two inequalities together and rearranging yields
$|c-\varphi|^2L^{\lambda p} \leq1$, which implies the upper
estimate in (a).

\begin{sublemma}\label{expand}
We have
 $|c_{p}-f^{p+1}\varphi|\geq \mathcal O(1)L^{2-\beta-3\alpha
p}$.\end{sublemma}
\begin{proof} We have
\begin{align*}
|c_{p}-f^{p+1}\varphi|
&\geq \mathcal O(1)|(f^{ p})'c_0|D_{p+1}(c_0)\\
&= \mathcal O(1)L^{-\beta}\left[\sum_{i=0}^{p}|(f^{
p})'c_0|^{-1}d_i^{-1}\right]^{-1}.\end{align*} For the sum in the
square bracket, $|(f^{p})'c_0|d_i\geq L^2
\left(\min\{\sigma,L^{-\alpha i}\}\right)^2$ which follows from
${(X)}_{n,c}$ gives
\begin{align*}
\sum_{i=0}^{p}|(f^{p})'c_0|^{-1}d_i^{-1} &\leq
L^{-2}\sigma^{-2}\cdot\frac{-\log\sigma}{\alpha\log L
}+L^{-2}\sum_{i=0}^{p}L^{2\alpha i}\\
&\leq L^{-2+3\alpha p}.\end{align*} This implies the desired
inequality.
\end{proof}

Sublemma \ref{expand} and $|f'|\leq K_0L$ give
$$L^{-3\alpha p}\leq|c_p-f^{p+1}\varphi|\leq \mathcal O(1)L|c-\varphi|^2(K_0L)^p.$$
Rearranging this yields the lower estimate of $p$ in (a).

We have \begin{align*} |(f^{p+1})'\varphi|&\geq \mathcal
O(1)L|(f^p)'c_0||c-\varphi|\\&\geq \mathcal
O(1)|(f^p)'c_0|D_{p+1}(c_0)|c-\varphi|^{-1}.
\end{align*} Sublemma \ref{expand} gives
$|(f^p)'c_0|D_{p+1}(c_0)|\geq L^{2-\beta-3\alpha p}.$ Hence
$$|(f^{p+1})'\varphi|\geq \mathcal O(1)L^{2-\beta-3\alpha p}
|c-\varphi|^{-1}.$$ Substituting into this the upper estimate of
$p$ in (a) gives
$$
|(f^{p+1})'\varphi|\geq
L^{2-\beta}|c-\varphi|^{-1+\frac{6\alpha}{\lambda}}.$$
Substituting $|c-\varphi|^{-1}\geq L^{\lambda p/2}$ into this
which follows from (a),
$$
|(f^{p+1})'\varphi|\geq
L^{2-\beta}|c-\varphi|^{-1+\frac{6\alpha}{\lambda}} \geq
L^{\lambda(p+1)/3}.$$  The last inequality is a consequence of
$2-\beta>\lambda_0$. This proves (b). A proof of (c) is analogous
to that of (b).
\end{proof}

\subsection{Decompositions into "bound" and "free" segments.} We
introduce some useful language along the way. For $\theta\in S^1$
such that $f^i\theta\notin C$ for all $i\geq1$, let
$$n_1<n_1+p_1+1\leq n_2<n_2+p_2+1\leq\cdots$$
be defined as follows: $n_1$ is the smallest $j\geq0$ such that
$f^j\theta\in C_\delta$. For $k\geq1$, let $p_k$ be the bound
period of $f^{n_k}\theta$, and let $n_{k+1}$ be the smallest
$j\geq n_k+p_k+1$ such that $f^j\theta\in C_\delta$. (Note that an
orbit may return to $C_\delta$ during its bound periods, i.e.
$n_i$ are not the only return times to $C_\delta$.) This
decomposes the orbit of $\theta$ into segments corresponding to
time intervals $(n_k,n_k+p_k]$ and $[n_k+p_k+1,n_{k+1}]$, during
which we describe the orbit of $\theta$ as being "bound" and
"free" states respectively; $n_k$ are called times of {\it free
returns}. For orbits which return to $C_{\delta_0}$ but not to
$C_\delta$, we similarly define bound and free states using
$C_{\delta_0}$ instead of $C_{\delta}$.

The next lemma asserts that no return to $C_{\delta_0}$ occurs
during these bound periods.

\begin{lemma}\label{no}
Let $\theta$ make a free return at time $n$ to $C_{\delta_0}
\setminus C_\delta$, with $p$ the corresponding bound period. Then
$f^{i}\notin C_{\delta_0}$ holds for every $i\in[n+1,n+p+1]$.
\end{lemma}
\begin{proof}
Let $c$ denote the critical point to which $f^{n}\theta$ is bound.
Lemma \ref{dist} gives $|f^{i}\theta-f^{i-n}c|\leq KK_0
|(f^{i-n})'(c_0)| D_{p}(c_0) \leq L^{-\beta}.$ Hence
\begin{align*}
d(f^{i}\theta,C)&\geq  d(f^{i-n}c,C)-|f^{i}\theta-f^{i-n}c|\\
&\geq\sigma-L^{-\beta}\quad\text{by $i-n\leq N$ and (\ref{mis})}\\
&>\delta_0.
\end{align*}
This yields the claim.
\end{proof}

\subsection{Exponential growth outside of critical neighborhoods}
The next corollary asserts an exponential growth of derivatives
outside of $C_\delta$. Our derivation of this consists of two
explicit parts: the obvious exponential growth outside of
$C_{\delta_0}$ for free segments and the recovered expansion in
Lemma \ref{reclem1} for bound segments associated with returns to
$C_{\delta_0}$.

\begin{cor}\label{outside}
For any $a\in A^{(N)}$, $f=f_{a}$ satisfies the following:
\begin{itemize}
\renewcommand{\labelitemi}{$\text{(a)}$}
\item if $n\geq1$ and $\theta$, $f_a\theta,\cdots,
f_a^{n-1}\theta\notin C_\delta$, then
  $|(f_a^{n})'\theta|\geq L^{1-3\lambda}\delta L^{3\lambda n}$;
\renewcommand{\labelitemi}{$\text{(b)}$}
\item if moreover $f_a^n\theta\in C_\delta$,  then $|(f_a^{n})'\theta|
\geq L^{3\lambda n}.$
\end{itemize}
\end{cor}
\begin{proof}
If the orbit of $\theta$ makes no return to $C_{\delta_0}$ in the
first $n-1$ iterates, then the assertions are obvious. Otherwise,
let $0<n_1<\cdots<n_s\leq n-1$ denote the sequence of all free
returns to $C_{\delta_0}$ in the first $n-1$ iterates of $\theta$,
with $p_1,\cdots,p_s$ the corresponding bound periods. We have
\begin{align*}
|(f^{n_s-n_1})'f^{n_1}\theta|=\prod_{1\leq k\leq
s-1}|(f^{n_{k+1}-n_k-p_k-1})'f^{n_k+p_k+1}
\theta||(f^{p_k+1})'f^{n_k} \theta|.
\end{align*}
For each bound segment, (c) in Lemma \ref{reclem1} gives
$|(f^{p_{k}+1})'f^{n_k} \theta|\geq L^{ (p_{k}+1)\lambda_0/3}.$
For each free segment, we clearly have
$|(f^{n_{k+1}-n_k-p_k-1})'f^{n_k+p_k+1} \theta|\geq
L^{\lambda_0(n_{k+1}-n_k-p_k-1)}$. Hence
\begin{align*}
|(f^{n_s-n_1})'f^{n_1}\theta|&\geq\prod_{1\leq k\leq
s-1}L^{\lambda(n_{k+1}-n_k)}\geq L^{\frac{\lambda_0}{3}(n_s-n_1)}.
\end{align*}
Since $|(f^{n_1})'\theta|\geq L^{\lambda_0 n_1}$ we obtain
$|(f^{n_s})'\theta|\geq L^{\frac{\lambda_0 n_s}{3}}.$

For the remaining factor, by Lemma \ref{no} we have
$f^i\theta\notin C_{\delta_0}$ for $n_s+1\leq i\leq n-1$. This and
$d(f^{n_s}\theta,C)\geq\delta$ gives
$$|(f^{n-n_s})'f^{n_s}\theta|\geq K_0^{-1}L\delta
L^{\lambda_0(n-n_s-1)}\geq\delta L^{\lambda_0(n-n_s)}\times
L^{1-\lambda}.$$ Consequently we obtain (a). If $f^n\in C_\delta$,
then Lemma \ref{no} gives $n\geq n_s+p_s+1$, and thus
$|(f^{n-n_s})'f^{n_s}\theta|\geq L^{\frac{\lambda_0}{3}(n-n_s)}.$
This proves (b).
\end{proof}

\section{Recovering inductive assumptions}
We have already defined the sets $A^{(0)},\cdots, A^{(N)}$ and
estimated their measure. At step $n\geq N$, we exclude from
$A^{(n)}$ all parameters for which $X_{n+1}$ or $Y_{n+1}$ may
fail. We introduce condition $W_n$ which determines a rule of
exclusion. Parameters have to satisfy this condition to be
selected. In other words, for $n\geq N$ we define
$$A^{(n+1)}=\{a\in A^{(n)}\colon W_n\text{ holds}\}.$$

\subsection{Condition ${W}_{n}$}
Let $f_a$ satisfy ${X}_{n}$, ${Y}_{n}$. We say $f_a$ satisfies
${W}_{n,c}$ if
\begin{equation}\label{br}\sum_{\stackrel{i\leq k}{\text{free return}}}-\log d
(c_i(a),C)\leq\frac{\alpha k\log L}{3}\text{ for every
}k\in[0,n].\end{equation} The sum runs over all $i$ at which the
orbit of $c_0$ makes a free return to $C_\delta$. We say $f_a$
satisfies $(W)_n$ if it satisfies $(W)_{n,c}$ for every $c\in C$.

The next proposition implies that the assumptions of the induction
are recovered for parameters in $A^{(n+1)}$.
\begin{prop}\label{brprop}
Let $f_a$ satisfy ${W}_{n,c}$. Then ${X}_{n+1,c}$ and ${
Y}_{n+1,c}$ hold for $f_a$.
\end{prop}

\begin{proof}
We begin with the particular case in which $c_0$ makes no return
to $C_\delta$ in the first $n$ iterates. (\ref{mis}) gives
$\delta|(f^{N})'c_0|\geq \delta(K_0^{-1}L\sigma)^{N}\geq
L^{3\lambda N}.$ Hence
\begin{align*}
|(f^{n+1})'c_0|&=|(f^{n+1-N})'c_{N}||(f^{N})'c_0|\\
&\geq \delta L^{3\lambda(n+1-N)}|(f^{N})'c_0|
\quad\text{ by Corollary \ref{outside}}\\
&\geq L^{3\lambda(n+1)},
\end{align*}
which proves ${Y}_{n+1,c}$.

To check $X_{n+1,c}$ we only need to consider inequalities which
are not covered by $X_{n,c}$. They are:
\begin{equation}\label{ba}
|(f^{n+1-i})'c_i|\geq L\cdot\min\{\sigma,L^{-\alpha i}\}
\quad\text{ for }0\leq i\leq n.\end{equation}

Suppose that $i\geq N$. Corollary \ref{outside} and the definition
of $\delta$ in (\ref{delta0}) give
\begin{align*}|(f^{n+1-i})'c_i|
\geq L\delta \geq LL^{-\alpha i}.\end{align*} Suppose that $i<N$.
Since $c_i,c_{i+1},\cdots,c_{i+N}\notin C_\delta$, Corollary
\ref{outside} gives $|(f^{N})'c_i|\geq \delta L^{\lambda N}$ and
$|(f^{n+1-i-N})'c_{N+i}|\geq L\delta$. Therefore
\begin{align*}
|(f^{n+1-i})'c_i|=|(f^{n+1-i-N})'c_{N+i}||(f^{N})'c_i|\geq L,
\end{align*}
where the last inequality follows from the definition of $\delta$
and $\alpha\ll\lambda$.

Proceeding to the general case, let $0<n_1<\cdots<n_s\leq n$
denote all the free returns to $C_\delta$ in the first
$n$-iterates of $c_0$, with $p_1,\cdots,p_s$ the corresponding
bound periods. Using Lemma \ref{reclem1} and then $W_{n,c}$,
\begin{equation}\label{p}\sum_{k=1}^{s}p_k\leq
\frac{2}{\lambda\log L}\sum_{k=1}^{s}-\log d(c_{n_k},C)\leq
\frac{2\alpha n }{3\lambda\log L}.\end{equation} The chain rule
gives
\begin{align*}
|(f^{n_{s}+p_{s}+1})'c_0|&=|(f^{n_1})'c_0|\cdot\prod_{k=1}^{s-1}
|(f^{n_{k+1}-n_{k}-p_{k}-1})'c_{n_{k}+p_{k}+1} |\prod_{k=1}^{s}
|(f^{p_{k}+1})'c_{n_{k}}|.\end{align*} For the first term,
$\lambda\gg\alpha$ and $n_1\geq N$ give $L^{\lambda
n_1}\geq\delta^{-1}$. This and Corollary \ref{outside} give
$|(f^{n_1})'c_0|\geq L^{\lambda n_1}L^{2\lambda n_1}\geq
\delta^{-1}L^{2\lambda n_1}.$ Using Lemma \ref{reclem1} for each
term in the products, we obtain
\begin{equation}\label{breq1}
|(f^{n_{s}+p_{s}+1})'c_0|\geq
\delta^{-1}L^{2\lambda(n_{s}+p_{s}+1-\sum_{k=1}^{s} p_k)}.
\end{equation}

The rest of the argument splits into two cases. First, suppose
that $n_s+p_s\geq n$. Using $|f'|\leq K_0L$ and
$p_s\leq\frac{2\alpha n}{3\lambda\log L}$ in (\ref{p}) we have
$$\frac{|(f^{n_{s}+p_{s}+1})'c_0|}{|(f^{n+1})'c_0|}\leq (K_0L)^{p_s}
\leq L^{\alpha n}.$$ Hence
\begin{equation}\label{breq2}|(f^{n+1})'c_0|\geq L^{2\lambda(n_{s}+p_{s}+1-\sum_{k=1}^{s}
p_k)} \geq L^{2\lambda(n+1-\sum_{k=1}^{s} p_k)}.\end{equation}

Next, suppose that $n_s+p_s< n$. Since $n_s$ is the last free
return, Corollary \ref{outside} gives
$|(f^{n+1-n_{s}-p_{s}-1})'c_{n_{s}+p_{s}+1}|\geq K_0^{-1}\delta
L^{3\lambda(n-n_{s}-p_{s})}.$ Combining this with (\ref{breq1})
gives \begin{equation} \label{breq3} |(f^{n+1})'c_0|\geq
L^{2\lambda(n+1-\sum_{k=1}^{s} p_k)},\end{equation} yielding the
same inequality as in the previous case. In either of these two
cases, substituting the upper estimate of the sum of the bound
periods in (\ref{p}) into the exponent of (\ref{breq3}) yields
the desired inequality.\\

\noindent{\it Proof of ${X}_{n+1,c}$.} We deal with four cases
separately.

Case (i): $i,n+1\notin \cup_{1\leq k\leq s} [n_k+1,n_k+p_k]$. We
have
 \begin{align*}|(f^{n+1-i})'\theta_i|&\geq
\delta L^{\lambda(n+1-i)}\quad\text{by Corollary \ref{outside} and
Lemma \ref{reclem1} }\\&\geq \delta L^{\alpha(n+1)-\alpha
i}\quad\text{since $\alpha<\lambda$}\\
&\geq L^{1-\alpha i}\quad\text{since $\delta L^{\alpha(n+1)}\geq
L\delta L^{\alpha N}\geq L$}.\end{align*}

Case (ii): $i\notin \cup_{1\leq k\leq s} [n_k+1,n_k+p_k]$ and
$n+1\in[n_s+1,n_s+p_s]$. Let $\tilde c$ denote the critical point
to which $c_{n_s}$ is bound. We have:
\begin{align*}
&d(c_{n_s},C)\geq L^{-\frac{\alpha n}{3}}\quad\text{by
$W_{n,c}$,}\\
&|(f^{n-n_s})'c_{n_s+1}|\geq K^{-1}L^{\lambda(n-n_s)}
\quad\text{by Lemma \ref{dist} and $Y_{n,\tilde c}$,}\\ &
|(f^{n_s-i})'c_{i}|\geq L^{\frac{\lambda}{3}(n_s-i)}
\quad\text{since $i$, $n_s$ are free and $n_s$ is a return.}
\end{align*}
Combining these altogether gives
\begin{align*}|(f^{n+1-i})'c_i|&\geq
|(f^{n-n_s})'c_{n_s+1}|K_0^{-1}Ld(c_{n_s},C)|(f^{n_s-i})'c_{i}|\\&
\geq K_0^{-1}K^{-1}L^{\frac{\lambda}{3}(n-i)+1}L^{-\alpha
n/3}\\&\geq
K_0^{-1}K^{-1}L^{\alpha(n-i)/3+1}L^{-\alpha n/3}\\
&\geq L^{1-\alpha i/2},
\end{align*}
where the last inequality holds because $i\geq1$ and $K\to1$ as
$L\to\infty$.

Case (iii): $i,n+1\in[n_k+1,n_k+p_k]$ for some $k\in[1,s]$. Let
$\tilde c$ denote the critical point to which $c_{n_k}$ is bound.
If $\sigma\geq L^{-\alpha (i-n_k-1)}$, then
\begin{align*}
|(f^{n+1-i})'c_i|&\geq K^{-1}|(f^{n+1-i})'
\tilde c_{i-n_k-1}|\quad\text{by Lemma \ref{dist}}\\
&\geq K^{-1}L\cdot L^{-\alpha (i-n_k-1)}\quad\text{by
$X_{n,\tilde c}$}\\
&\geq L\cdot L^{-\alpha i}.\end{align*}

Suppose that $\sigma< L^{-\alpha (i-n_k-1)}$. Then the definition
of $\sigma$ in (\ref{sigma}) gives $i-n_k\leq 2\alpha^{-1}$. If
$n-i\leq \alpha N$, then the bound period $p_k$ for $n_k$ remains
in effect at time $n$, because all bound periods are $\geq\alpha
N$ (Lemma \ref{reclem1}). Hence
\begin{align*}|(f^{n+1-i})'c_i|&\geq
K^{-1}|(f^{n+1-i})' \tilde c_{i-n_k-1}|\\
&\geq 1\quad\text{by (\ref{mis}).}\end{align*} It is left to
consider the subcase $n-i\geq \alpha N$. The proof of (ii) and
$n_k<i$ imply
$$|(f^{n+1-n_k})'c_{n_k}|\geq
L^{\frac{\lambda}{3}(n-n_k)}L^{1-\alpha n_k/2}\geq
L^{\frac{\lambda}{3}(n-i)} L^{1-\alpha i/2}.$$ Combining this with
$|(f^{i-n_k})'c_{n_k}|\leq(K_0L)^{-2\alpha^{-1}}$ we obtain
\begin{align*}
|(f^{n+1-i})'c_{i}|&=|(f^{n+1-n_k})'c_{n_k}|
|(f^{i-n_k})'c_{n_k}|^{-1}\\
&\geq L^{\frac{\lambda N}{3}} L^{1-\alpha i/2}
(K_0L)^{-2\alpha^{-1}}\\
&\geq L^{1-\alpha i/2}.
\end{align*}

Case (iv): $i\in[n_k+1,n_k+p_k]$ for some $k\in[1,s]$. We may
assume $n_k+p_k<n+1$, for otherwise the case is covered by (iii).
We have
\begin{align*} |(f^{n-n_k-p_k})'c_{n_k+p_k+1}|&\geq
L^{1-\alpha(n_k+p_k+1)/2}\quad\text{by (ii)}\\
&\geq L^{1-2\alpha n_k/3}\quad\text{by $p_k\leq\alpha n_k$ from
$W_{n,c}$}.
\end{align*}

To estimate the remaining factor, we consider two cases separately
as before. Suppose that $\sigma< L^{-\alpha(i-n_k-1)}$, then
$i-n_k\leq 2\alpha^{-1}$.
\begin{align*}|(f^{n_k+p_k+1-i})'c_i|&=
|(f^{p_k+1})'c_{n_k}||(f^{i-n_k})'c_{n_k}|^{-1}\quad\text{by Lemma \ref{dist}}\\
&\geq L^{\lambda p_k}(K_0L)^{-2\alpha^{-1}}\\
&\geq1\quad\text{since $p_k\geq\alpha N$.}\end{align*} Combining
this with the previous inequality implies the desired one.

Suppose that $\sigma\geq L^{-\alpha(i-n_k-1)}$. Let $\tilde c$
denote the critical point to which $c_{n_k}$ is bound. Then
\begin{align*}|(f^{n_k+p_k+1-i})'c_i|&\geq
K^{-1}|(f^{n_k+p_k+1-i})'\tilde c_{i-n_k-1}|\quad\text{by Lemma \ref{dist}}\\
&\geq K^{-1}L^{1-\alpha(i-n_k-1)}\quad\text{by $X_{n,\tilde
c}$}\end{align*} Combining these two inequalities,
$$|(f^{n+1-i})'c_i|\geq L^{2-\alpha i+\alpha+\alpha n_k/3}
\geq L^{1-\alpha i}.$$ This completes the proof of Proposition
\ref{brprop}.
\end{proof}

\section{Parameter exclusion: general steps}
Set $A^{(\infty)}=\bigcap_{n\geq0}A^{(n)}$. In this last section
we prove

\begin{prop}\label{measure}
For all large integer $N$ there exists $L_0$ such that for all
$L\geq L_0$,
$$ \left|A _L^{(\infty)}\right|
\geq \left(1-L^{-\frac{\alpha
N}{10}}\right)\left(1-{^3\sqrt{\sigma}}\right)^N.$$
\end{prop}
We prove this by combining the estimate of $|A^{(N)}|$ in
Proposition \ref{initial} with a new estimate which shows that the
measure of $A^{(n)}\setminus A^{(n+1)}$ relative to $|A^{(N)}|$
decreases exponentially in $n$.

\subsection{Expansion in parameter space}
We begin with a main technical estimate. Let $c\in C$, $a\in
A^{(n)}$ and suppose that $c_0(a)$ makes a free return to
$C_\delta$ at time $\nu\leq n+1$. We say $\nu$ is an {\it
essential return time} if
\begin{equation}\label{inessential}
\sum_{\stackrel{i+1\leq j\leq\nu}{\text{free return}}}2\log
d(c_{j}(a),C)\leq \log d(c_{i}(a),C)\text{ for every free return
}i\leq[0,\nu-1].
\end{equation}
The sum ranges over all $j\in[i+1,\nu]$ at which $c_0(a)$ makes a
free return to $C_\delta$.

\begin{lemma}\label{expansion}
Let $a\in A^{(n)}$ and $c\in C$. For every essential return $\nu$
in the first $n+1$ iterates of $c_0(a)$ we have
$$D_\nu(a,c_0(a))\cdot|(f_a^{\nu})'c_0(a)|\geq
\sqrt{d(c_\nu,C)}.$$ In particular, for all $b\in
\Delta_{\nu}(a,c)-
{^{5}\sqrt{d(c_\nu(a),C)}}\cdot\Delta_{\nu}(a,c)$ we have
$$|c_\nu(a)-c_\nu(b)| \geq
{^4\sqrt{d(c_\nu(a),C)}}.$$
\end{lemma}
\begin{proof}
We finish the proof of the second assertion assuming the
first
 one.
The assumption on $b$ gives $b-a\geq\frac{1}{2}\cdot
{^{5}\sqrt{d(c_\nu(a),C)}}\cdot|\Delta_{\nu}(a,c)|.$ Since
$|\Delta_{\nu}(a,c)|\geq \frac{|\hat
\Delta_{\nu}(a,c))|}{9|c_n(\hat \Delta_{\nu}(a,c))|}$ and
$|c_n(\hat \Delta_{\nu}(a,c))|\leq \mathcal O(1)L^{2-\beta}$ by
Sublemma \ref{distrem1}, we have \begin{align*} b-a &\geq \mathcal
O(1)L^{\beta -2}\cdot {^{5}\sqrt{d(c_\nu(a),C)}}\cdot|\Delta_{\nu}(a,c)|\\
&\geq {^{4}\sqrt{d(c_\nu(a),C)}}D_\nu(a,c_0(a)),
\end{align*}
where the last inequality is because of $d(c_\nu(a),C)\leq\delta$.
Hence we have
\begin{align*}
|c_{\nu}(a)-c_{\nu}(b)| &\geq {K'}^{-1}
\left|c_\nu'(a)\right||a-b| \quad\text{by
Lemma \ref{samp}}\\
&\geq \mathcal O(1)|(f_a^\nu)'c_0(a)||a-b|\quad\text{by
Lemma \ref{trans}}\\
&\geq\ \mathcal O(1)L^{-4}|(f_{a}^{\nu})'c_0(a)|D_\nu(a,c_0(a))
\cdot{^5\sqrt{d(c_\nu(a),C)}}\\
&\geq{^4\sqrt{d(c_\nu(a),C)}}\quad\text{by the first
assertion}.\end{align*}

Let $0<n_1<\cdots<n_t<\nu$ denote all the free returns in the
first $\nu$ iterates of $c_0(a)$, with $p_1,\cdots,p_t$ the
corresponding bound periods. Let
$$S_{n_k}=\sum_{i=n_k}^{n_k+p_k}d_i^{-1}(c_0)\text{ and }
S_{0}=\sum_{i=0}^{\nu-1}d_i^{-1}(c_0)-\sum_{k=1}^{t} S_{n_k}.$$

\begin{sublemma}\label{quatro}
We have $ S_0|(f^\nu)'c_0|^{-1}\leq \frac{1}{\sqrt\delta}.$
\end{sublemma}
\begin{proof}

Let $i\notin\cup_{1\leq k\leq t}[n_k,n_k+p_k]$. Suppose $c_i\notin
C_{\sqrt{\delta}}$. We have $|f'c_i|\geq K_0^{-1}L\sqrt\delta$.
Split the itinerary from time to $i$ to $\nu$ into bound and free
segments. Using Lemma \ref{reclem1} to each bound segment and
Corollary \ref{outside} to each free segment, we have
$|(f^{\nu-i})'c_i|\geq L^{\lambda(\nu-i)/3}.$ Hence
$$|(f^\nu)'c_0|d_i=|(f^{\nu-i})'c_i| |f'c_i|\geq
L^{\lambda(\nu-i)/3}K_0^{-1}\sqrt\delta.$$

Suppose $c_i\notin C_{\sqrt{\delta}}$. Let $p$ denote the
corresponding bound period. By $c_i\notin C_{\delta}$ and Lemma
\ref{no}, no bound return follows $p$. Thus, we can estimate
$|(f^{\nu-i})'c_i|$ by split the itinerary of $c_0$ into free and
bound states and argue in the same way as the previous case. Thus
we have
$$|(f^{\nu-i})'c_i| |f'c_i|\geq
(L^{\lambda/3})^{\nu-i-p-1} |(f^{p+1})'c_i||f'c_i|.$$ Lemma
\ref{reclem1} gives
\begin{align*}|(f^{p+1})'c_i||f'c_i|&=^{10}\sqrt{|(f^{p+1})'c_i|}
\times ^{10}\sqrt{|(f^{p+1})'c_i|^9}|f'c_i|\geq
{^4{\sqrt{\delta}}}\cdot L^{\frac{\lambda(p+1)}{10}} .\end{align*}
These inequalities altogether yield $|(f^\nu)'c_0|d_i\geq
{^4{\sqrt{\delta}}}L^{\frac{\lambda(\nu-i)}{10}}.$

Combining the above two estimates give
\begin{equation*}
S_0|(f^\nu)'c_0|^{-1}\leq {\frac{1}{^{4}{\sqrt\delta}}}
\sum_{i=0}^{\nu-1} L^{-\frac{(\nu-i)}{10}}+
 {\frac{1}{{\sqrt\delta}}}\sum_{i=0}^{\nu-1}
L^{-\lambda(\nu-i)/3}\leq {\frac{1}{{\sqrt\delta}}}
.\end{equation*} \end{proof}

\begin{sublemma}\label{sublem2}
For every $1\leq k\leq t$ we have
$S_{n_k}|(f^{n_k+p_k+1})'c_0|^{-1} \leq
|d(c_{\nu_k},C)|^{-\frac{3\alpha}{\lambda}}.$
\end{sublemma}
\begin{proof}
Let $\tilde c$ denote the critical point to which $c_{n_k}$ is
bound. By Lemma \ref{dist}, for $n_k+1\leq i\leq n_k+p_k$ we have
\begin{align*}|(f^{n_k+p_k+1})'c_0|
d_{i}(c_0)&=\frac{|(f^{n_k+p_k+1})'c_0|}
{|(f^{i+1})'c_{0}|}\frac{|(f^{i+1})'c_0|} {|(f^{i})'c_{0}|}\\&
\geq K^{-4}\frac{|(f^{p_k})'\tilde c_0|} {|(f^{i-n_k})'\tilde
c_{0}|}\frac{|(f^{i-n_k})'\tilde c_0|} {|(f^{i-n_k-1})'\tilde
c_{0}|}.
\end{align*}
Using $(X)_{n,\tilde c}$ to estimate the two fractions we
have
$$|(f^{n_k+p_k+1})'c_0| d_{i}(c_0)\geq
K^{-4}L^2\left(\min\{\sigma,L^{-\alpha(i-n_k)}\}\right)^2.$$
Proposition \ref{reclem1} gives $|(f^{n_k+p_k+1})'c_0|d_{n_k}\geq
|d(c_{n_k},C)|^{\frac{\alpha}{\lambda}}$, and therefore
\begin{align*}
S_{n_k}|(f^{n_k+p_k+1})'c_0|^{-1}&\leq|d(c_{n_k},C)|^{-
\frac{\alpha}{\lambda}}+
K^4L^{-2}\sum_{i=n_k+1}^{n_k+p_k}\left(\min\{\sigma,L^{-\alpha(i-n_k)}\}\right)^{-2}\\
&\leq
|d(c_{n_k},C)|^{-\frac{\alpha}{\lambda}}-\sigma^{-1}\alpha^{-1}\log\sigma
+\sum_{i=0}^{p_k} L^{\alpha i}\\
&\leq |d(c_{n_k},C)|^{-\frac{3\alpha}{\lambda}}.\end{align*} For
the last inequality, we have used the upper estimate of $p_k$ in
Lemma \ref{reclem1} to estimate the last term.
\end{proof}

Write $\nu=n_{t+1}$. Since $n_{t+1}$ is an essential return we
have
\begin{equation*}|d(c_{n_k},C)|^{-1}\leq \prod_{k+1\leq j\leq t+1}
|d(c_{n_j},C)|^{-2} \quad\text{ for every }1\leq k\leq
t.\end{equation*} Substituting this into the inequality in
Sublemma \ref{sublem2} gives
\begin{equation}\label{plu}
S_{n_k}|(f^{n_k+p_k+1})'c_0|^{-1}\leq \prod_{k+1\leq j\leq
t+1}|d(c_{n_j},C)|^{-\frac{6\alpha}{\lambda}}.\end{equation}
Meanwhile, for $1\leq k\leq t-1$ we have
\begin{align*}
|(f^{\nu-n_k-p_k-1})'c_{n_k+p_k+1}|&=\prod_{k\leq j\leq t
}|(f^{n_{j+1}-n_j-p_j-1})'c_{n_j+p_j+1}| \cdot\prod_{k+1\leq
u\leq t} |(f^{p_j+1})'c_{n_j}|\\
&\geq\prod_{k+1\leq j\leq t} |(f^{p_j+1})'c_{n_j}|.\end{align*}
Multiplying (\ref{plu}) with the above inequality gives
\begin{align*}
S_{n_k}|(f^{\nu})'c_0|^{-1}&\leq
|d(c_{\nu},C)|^{-\frac{6\alpha}{\lambda}}\prod_{k+1\leq j\leq
t}\frac{1}{|d(c_{n_j},C)|^{\frac{6\alpha}{\lambda}}|(f^{p_j+1})'c_{n_j}|}\\
&\leq\delta^{(t-k)/2}|d(c_{\nu},C)|^{-\frac{6\alpha}{\lambda}}\quad\text{by
Lemma \ref{reclem1}}.\end{align*} Summing this over all $1\leq
k\leq t-1$ and (\ref{plu}) for $k=t$ gives
\begin{align*}\sum_{1\leq k\leq t}S_{n_k}|(f^{\nu})'
c_{0}| ^{-1}&\leq |d(c_{\nu},C)|^{-\frac{6\alpha}{\lambda}}\left(
1+\sum_{k=1}^{t-1}\delta^{(t-k)/2}\right)\leq
2|d(c_{\nu},C)|^{-\frac{6\alpha}{\lambda}}.
\end{align*}
Combining this with Sublemma \ref{quatro} we obtain
\begin{align*}|(f^{\nu})'c_{0}| ^{-1}D_\nu^{-1}=
L^\beta\left( \sum_{k=1}^{t}S_{n_k}|(f^{\nu})'c_{0}| ^{-1}+
S_0|(f^{\nu})'c_{0}| ^{-1}\right)\leq
\frac{1}{\sqrt{d(c_{\nu},C)}}.
\end{align*}
This completes the proof of Lemma \ref{expansion}.
\end{proof}

\subsection{Strategy}
The estimate of the measure of $A^{(n)}\setminus A^{(n+1)}$
consists of three steps. We first decompose the set into a finite
number of subsets which are characterized by combinatorial data
describing the recurrence to the critical points. Then we estimate
the measure of each of the subsets. Finally we put these estimates
together, counting the total number of combinations.

\subsection{Decomposition of the parameter set}
For $c\in C$, $q\geq1$ and $n\geq N$, let
$$B_q(c)=\left\{a\in A^{(n)}\setminus A^{(n+1)}\colon\text{  $c_0(a)$ makes exactly $q$ e.f.r.s in the first $n$ iterates}\right\}.$$
We have $A^{(n)}\setminus A^{(n+1)} \leq \bigcup_{c\in
C}\bigcup_{q\geq1}B_q(c).$ We further decompose the right hand
side as follows. For a $q$-tuple of pairs of positive integers
$\mathcal X= ((\nu_1,\tau_1),\cdots,(\nu_{q},\tau_q))$, let
$B_{\mathcal X}(c)$ denote the set of all $a\in B_q(c)$ such that
$c_0(a)$ makes essential returns exactly at
$0<\nu_1<\cdots<\nu_{q}=n$, with
$|c_{\nu_i}(a)-c^{\tau_i}|\leq\delta$ for every $1\leq i\leq q$.
In other words, $c_{\nu_i}(a)$ is bound to $c^{\tau_i}$. We write
$B_{\mathcal X}(c)= \bigcup_RB_{\mathcal X }^R(c)$, where
$$B_{\mathcal X}^R(c)=\left\{x\in B_{\mathcal X}(c)\colon
\sum_{i=1}^{q}\left[-\log d(c_{\nu_i}(a),C)\right]=R\right\},$$
where the square bracket denotes the integer part.
 For each partition $\mathcal Y= (r_1,\cdots,r_{q})$ of $R$ into $q$
 positive integers, define
$$B_{\mathcal X,\mathcal Y}^{R}(c)=\left\{a\in
B_{\mathcal X}^{R}(c)\colon \left[ -\log
d(c_{\nu_i}(a),C)\right]=r_i\text{ for every }1\leq i\leq
q\right\}.$$ This set is a finite union of intervals. We clearly
have $B_q(c)=\bigcup_{\mathcal X}\bigcup_{R}\bigcup_{\mathcal Y}
B_{\mathcal X,\mathcal Y}^{R}(c).$

\subsection{Hierarchical covering}
Let $\hat r_i=\exp\left({-\frac{r_i+1}{5}}\right)$.
 In order to estimate the measure of $B_{\mathcal
X,\mathcal Y}^{R}(c)$, we construct a {\it hierarchical covering}
of it, which consists of for each $i=1,2,\cdots,q$ a finite
collection of pairwise disjoint intervals $\{I^{(i)}_j\}$
intersecting $[0,1]$ such that:

$(i)$ $B_{\mathcal X,\mathcal Y}^{R}(c)\subset \cup _j \hat
r_i\cdot I^{(i)}_j$;

$(ii)$ for any $j$ there exists $k$ such that $I^{(i)}_j\subset
\hat r_{i-1}\cdot I^{(i-1)}_k.$

Then we clearly have $\sum_j|I_j^{(i)}|\leq \hat r_{i-1}\sum_j
|I_j^{(i-1)}|$ for each $i=q,q-1,\cdots,2$, and thus
$\sum_j|I_j^{(q)}|\leq \hat r_{1}\hat r_2\cdots\hat r_{q-1}\sum_j
|I_j^{(1)}|$. $(i)$ gives $|B_{\mathcal X,\mathcal Y}^{R}(c)|\leq
\hat r_q\sum_j|I_j^{(q)}|$, and thus $|B_{\mathcal X,\mathcal
Y}^{R}(c)|\leq \hat r _1\hat r _2\cdots\hat r
_q\sum_j|I_j^{(1)}|$. Since the intervals $\{I_j^{(1)}\}$ are
pairwise disjoint, intersect $[0,1)$, and $\leq L^{-\beta}$ in
length, we obtain
\begin{equation}
\label{meas1} |B_{\mathcal X,\mathcal Y}^{R}(c)|\leq
\exp\left(-\frac{R+q}{5}\right)\left(1+2L^{-\beta}\right).\end{equation}

For the construction of the hierarchical covering we need a couple
of lemmas. Let
$$\Delta_n(a,c,r)=\exp\left({-\frac{r}{5}}\right)\cdot\Delta_n(a,c).$$
\begin{lemma}\label{nest}
Let $a$, $b\in B_{\mathcal X,\mathcal Y}^{R}(c)$ and suppose
$\nu_i<\nu_j$. If $\Delta_{\nu_i}(a,c,r_i)\cap
\Delta_{\nu_j}(b,c,d)\neq\emptyset$ holds for some
$d\geq-\log\delta$, then $\Delta_{\nu_j}(b,c)\subset
\Delta_{\nu_i}(a,c,r_i-1)$.
\end{lemma}

\begin{proof}
By Lemma \ref{expansion}, there exists $\hat a\subset
\Delta_{\nu_i}(a,c,r_i)$ such that $c_{\nu_i}(\hat a)\in C$.
Proposition \ref{samp} and $\nu_i<\nu_j$ implies $\hat a\notin
\Delta_{\nu_j}(b,c)$, for otherwise ${\rm Dist}
(\gamma_{\nu_j},\Delta(b,\nu_j,0))$ is unbounded. This implies
that one of the connected components of
$\Delta_{\nu_j}(b,c)-\Delta_{\nu_j}(b,c,d)$ is contained in
$\Delta_{\nu_i}(a,c,r_i)$. This implies
$$2^{-1}(1-e^{-d/5})|\Delta_{\nu_j}(b,c)|\leq |\Delta_{\nu_i}
(a,c,r_i)|,$$ and hence the inclusion holds.
\end{proof}

\begin{lemma}\label{toyintersect}
Let $a$, $b\in B_{\mathcal X,\mathcal Y}^{R}(c)$. Assume that
$\Delta_{\nu_i}(a,c)\cap \Delta_{\nu_i}(b,c)\neq\emptyset$ and
$b\notin \Delta_{\nu_i}(a,c)$. Then we have $\Delta_{\nu_i}(a,c)
\subset \Delta_{\nu_i}(b,c)$.
\end{lemma}

\begin{proof}
Analogous to the proof of Lemma \ref{intersect2}.
\end{proof}

We are in position to construct a hierarchical covering of
$B_{\mathcal X,\mathcal Y}^{R}(c)$. First of all we claim that for
all $a\in B_{\mathcal X,\mathcal Y}^{R}(c)$ and each
$i=1,2,\cdots,q-1$ there exists a finite sequence $a_1,a_2,\cdots$
in $B_{\mathcal X,\mathcal Y}^{R}(c) \cap\Delta_{\nu_i}(a,c)$ such
that the corresponding intervals $\{\Delta_{\nu_{i+1}}(a_j,c)\}$
are (a) pairwise disjoint, (b) contained in
$\Delta_{\nu_i}(a,r_i-1)$, and (c) altogether cover $B_{\mathcal
X,\mathcal Y}^{R}(c) \cap\Delta_{\nu_i}(a,c)$.

This claim is proved as follows. We define the finite sequence in
the statement inductively as follows. Choose some $a_{1}\in
B_{\mathcal X,\mathcal Y}^{R}(c)\cap\Delta_{\nu_i}(a,c)$. Given
$a_1,\cdots,a_k$, we are done if $ B_{\mathcal X,\mathcal
Y}^{R}(c)\cap\Delta_{\nu_i}(a,c) \subset
\cup_{j=1}^{k}\Delta_{\nu_{i+1}}(a_j,c)$. Otherwise we choose some
$a_{k+1} \in B_{\mathcal X,\mathcal Y}^{R}(c)
\cap\Delta_{\nu_i}(a,c)-\cup_{j=1}^{k} \Delta_{\nu_{i+1}}(a_j,c).$
Since the length of the intervals $\{\Delta_{\nu_{i+1}}(a_j,c)\}$
are bounded from below, this definition makes sense and the
resultant finite number of intervals altogether cover $B_{\mathcal
X,\mathcal Y}^{R}(c)\cap\Delta_{\nu_i}(a,c).$ By Lemma
\ref{toyintersect}, any two of them are either disjoint or nested.
This proves (a) (c).


By Lemma \ref{expansion}, the set
$\Delta_{\nu_i}(a,c)-\Delta_{\nu_i}(a,c,r_i)$ does not intersect
$B_{\mathcal X,\mathcal Y}^{R}(c)$. Since $a_j\in B_{\mathcal
X,\mathcal Y}^{R}(c)$ and $a_j \in\Delta_{\nu_i}(a,c)$ we have
$a_j\in \Delta_{\nu_i}(a,c,r_i)$, which gives
$\Delta_{\nu_{i+1}}(a_j,c,r_{i+1})\cap
\Delta_{\nu_i}(a,c,r_i)\neq\emptyset$. This and Lemma \ref{nest}
yield (b).

We are in position to define the intervals $\{I^{(i)}_j\}$ for
each $i=1,2,\cdots,q$. Lemma \ref{toyintersect} implies the
existence of a finite sequence $a_1^{(1)},a_2^{(1)},\cdots$ in
$B_{\mathcal X,\mathcal Y}^{R}(c)$ such that the corresponding
intervals
$\Delta_{\nu_1}(a_1^{(1)},c),\Delta_{\nu_1}(a_2^{(1)},c),\cdots$
are pairwise disjoint and altogether cover $B_{\mathcal X,\mathcal
Y}^{R}(c).$ We set $I^{(1)}_j=\Delta_{\nu_1}(a_j^{(1)},c)$. Lemma
\ref{expansion} gives $B_{\mathcal X,\mathcal Y}^{R}(c)\subset
\bigcup_j \hat r_{1}\cdot I^{(1)}_j.$

Let $i\in[1,q-1]$. Having defined
$I^{(i)}_k=\Delta_{\nu_i}(a_k^{(i)},c)$ for $k=1,2,\cdots$, we
define the intervals $\{I^{(i+1)}_j\}$ which are contained in
$\hat r_i\cdot I^{(i)}_k$ as follows. According to the above
claim, there exists a finite set $a_1^{(i+1)},a_2^{(i+1)},\cdots$
in $B_{\mathcal X,\mathcal Y}^{R}(c)\cap I^{(i)}_k$ such that any
two of the corresponding intervals
$\{\Delta_{\nu_{i+1}}(a_j^{(i+1)},c)\}$ are pairwise disjoint,
contained in $\hat r_i\cdot I^{(i)}_k$, and altogether cover
$B_{\mathcal X,\mathcal Y}^{R}(c) \cap I^{(i)}_k$. We set
$\Delta_{\nu_{i+1}}(a_j^{(i+1)},c)=I^{(i+1)}_j$. By construction
we have $B_{\mathcal X,\mathcal Y}^{R}(c)\subset\bigcup_j
I^{(i+1)}_j,$ where the union runs over all intervals
$I^{(i+1)}_j$ defined in this way. Lemma \ref{expansion} gives
$B_{\mathcal X,\mathcal Y}^{R}(c)\subset \bigcup_j \hat
r_{i+1}\cdot I^{(i+1)}_j.$ This completes the construction of the
hierarchical covering. \qed

\subsection{Conclusion}
We are in position to finish the estimate of the measure of
$A^{(\infty)}$. We begin by counting all feasible $R$, $\mathcal
Y$, $\mathcal X$, $q$.

 The next lemma asserts that the the sum of essential
return depths has a definite proportion, and as a result gives a
lower bound for $R$.
\begin{lemma}
$\sum_{i=1}^q r_i\geq\alpha n\log L/2$.
\end{lemma}
\begin{proof}
We call a free return {\it inessential} if it is not an essential
return.
 Let $\mu\in(0,n)$ be an inessential return.
Let $i(\mu)$ denote the smallest $k\leq\mu-1$ such that
\begin{equation}\label{inessential}
\sum_{\stackrel{k+1\leq j\leq\mu}{\text{free return}}}2\log
d(c_j,C)> \log d(c_{k},C).
\end{equation}
No essential return occurs during the period $[i(\mu)+1,\mu]$. We
claim that there exists a set $F\subset(0,n)$ of inessential
returns such that the intervals $[i(\mu)+1,\mu]$ for $\mu\in F$
are mutually disjoint and cover all the inessential return times
in $(0,n)$. Indeed, it suffices to define
$F=\{\mu_1>\mu_2>\cdots>\mu_s\}$ as follows: let $\mu_1$ denote
the largest inessential return in $(0,n)$. Given $\mu_k\in F$, let
$\mu_{k+1}$ denote the largest inessential one which is $<
i(\mu_k)+1$.

Summing (\ref{inessential}) over all $\mu\in F$ gives
\begin{equation*}
\sum_{\stackrel{0\leq j\leq\nu}{\text{inessential}}}2\log
d(c_{j},C)> \sum_{\mu\in F}\log d(c_{i(\mu)},C)\geq
\sum_{\stackrel{0\leq i\leq n}{{\rm free}}}\log d(c_{i},C).
\end{equation*}
Therefore
$$\sum_{i=1}^q r_i=\sum_{\stackrel{0\leq j\leq\nu}
{\text{essential}}}-\log
d(c_{j},C)\geq-\frac{1}{2}\sum_{\stackrel{0\leq i\leq n}{{\rm
free}}}\log d(c_{i},C)\geq\frac{\alpha n\log L}{2}.$$
\end{proof}

Let $p_i$ denote the bound period for the orbit of $c$ at the free
return $\nu_i$. Let $\eta=\min\{p_1,p_2,\cdots,p_q\}$. Since $c$
is free at time $n$, we have
$$\eta (q-1)\leq \sum_{i=1}^{q-1}
p_i\leq n.$$ The upper estimate of the bound period in (a) Lemma
\ref{reclem1} gives
 $$\eta q\leq\sum_{i=1}^{q}
p_i\leq\frac{2R}{\lambda\log L}.$$ Since all the free returns
under consideration are to the inside of $C_\delta$, the lower
estimate in (a) of Lemma \ref{reclem1} gives $\eta\geq\alpha N$.
Hence, for sufficiently large $N$ we have
$$q/n,q/R\leq \frac{2}{\alpha N}.$$

By Stirling's formula for factorials, the number of all feasible
$\mathcal Y=(r_1,\cdots,r_q)$ is bounded by the total number of
combinations of dividing $R$ objects into $q$ groups,
which is $\left(\begin{smallmatrix}R+q\\
q\end{smallmatrix}\right)\leq e^{c(N)R}$ , where $c(N)\to0$ as
$N\to\infty$.  Similarly, the number of all feasible $\mathcal
X=((\nu_1,\tau_1), \cdots,(\nu_q,\tau_q))$ is $\leq
e^{c'(N)n}\cdot (\sharp C)^q \leq e^{c''(N)n},$ where $c''(N)\to0$
as $N\to\infty$.

For $n\geq N$ we have
\begin{align*}\left|A^{(n)}\setminus
A^{(n+1)}\right|&\leq \sum_{c\in C}\sum_q \sum_{\mathcal X
}\sum_{\mathcal Y} \sum_{R\geq\alpha n \log L/2}|B_{\mathcal X,\mathcal Y}^R(c)|\\
&\leq\sum_{c\in C}\sum_q \sum_{R\geq\alpha n\log L/2}
\exp\left(c(N)R+c''(N)n-\frac{{R+q}}{5}\right)(1+2L^{-\beta})\\
&\leq\sum_{R\geq\alpha n\log L/2}
\exp\left(c(N)R+c'''(N)n-\frac{{R}}{6}\right)(1+2L^{-\beta})\\
&\leq\sum_{R\geq\alpha n\log L/2} \exp\left(-\frac{{R}}{7}\right),
\end{align*}
where $c'''(N)\to0$ as $N\to\infty$. Substituting the estimate of
$|A^{(N)}|$ in Proposition \ref{initial} gives
\begin{align*}\frac{|A^{(n)}\setminus A^{(n+1)}|}
{|A^{(N)}|}\leq (1-{^3\sqrt\sigma})^{-N-1}\sum_R
\exp\left(-\frac{{R}}{7}\right)\leq L^{-\frac{\alpha n}{15}}
.\end{align*} Hence we obtain \begin{align*}
|A^{(\infty)}|&=|A^{(N)}|-\sum_{n\geq
N}|A^{(n)}\setminus A^{(n+1)}|\\
&\geq |A^{(N)}|\left(1-\sum_{n\geq N}L^{-
\frac{\alpha n}{15}}\right)\\
&\geq (1-{^3\sqrt\sigma})^{N+1}\left(1-\sum_{n\geq N}L^{-
\frac{\alpha n}{15}}\right).\end{align*} This completes the proof
of Proposition \ref{measure} and hence that of the main theorem.
\qed

\bibliographystyle{amsplain}

\end{document}